\documentclass[11pt]{amsart}

\usepackage{amsmath}
\usepackage{amsfonts}
\usepackage{amssymb}
\usepackage{amsthm}
\usepackage{indentfirst}
\usepackage[T1]{fontenc}

\usepackage{hyperref}

\usepackage[usenames]{color}

\usepackage{tikz}
\usepackage{verbatim}
\usetikzlibrary{matrix}

\usepackage{enumerate}

\hypersetup{
    colorlinks   = true,
      citecolor    = red
}

\hypersetup{linkcolor=red}

\hypersetup{linkcolor=blue}

\newtheorem{theorem}{Theorem}[section]
\newtheorem{thm}[theorem]{Theorem}
\newtheorem{lem}[theorem]{Lemma}
\newtheorem{prop}[theorem]{Proposition}

\newtheorem{defn}[theorem]{Definition}

\newcommand{\bG}{{\mathbb{G}}}

\newcommand{\A}{{\mathbb{A}}}

\newcommand{\N}{{\mathbb{N}}}
\newcommand{\Q}{{\mathbb{Q}}}

\newcommand{\Z}{{\mathbb{Z}}}
\newcommand{\cG}{{\mathcal{G}}}
\newcommand{\cO}{{\mathcal{O}}}
\newcommand{\cU}{{\mathcal{U}}}
\newcommand{\SL}{\operatorname{SL}}

\begin{document}

\begin{abstract}
In this paper, we formulate an analogue of Waring's problem for an algebraic group $G$.
At the field level we consider a morphism of varieties $f\colon \A^1\to G$
and ask whether every element of $G(K)$ is the product of a bounded number of elements
$f(\A^1(K)) = f(K)$.  We give an affirmative answer when $G$ is unipotent and $K$ is a characteristic zero field
which is not formally real.  

The idea is the same at the integral level, except one must work with schemes,
and the question is whether every element in a finite index subgroup of $G(\cO)$ can be written as a product of 
a bounded number of elements of $f(\cO)$.  We prove this is the case when $G$ is unipotent and $\cO$ is
the ring of integers of a totally imaginary number field.
\end{abstract}

\title{Waring's problem for unipotent algebraic groups}

\author{Michael Larsen}

\address{Department of Mathematics \\
          Indiana University \\
Bloomington, Indiana 47405, USA}

\email{\href{mailto:mjlarsen@indiana.edu}{\tt mjlarsen@indiana.edu}}

\author{Dong Quan Ngoc Nguyen}

\address{Department of Applied and Computational Mathematics and Statistics \\
         University of Notre Dame \\
         Notre Dame, Indiana 46556, USA }

\email{\href{mailto:dongquan.ngoc.nguyen@nd.edu}{\tt dongquan.ngoc.nguyen@nd.edu}}


\thanks{ML was partially supported by NSF Grant DMS-1401419.}

\maketitle

\section{Introduction}

The original version of Waring's problem asks whether, for every positive integer $n$ there exists $M:= M_n$ such that every non-negative integer is of the form $a_1^n+\cdots+a_M^n$, $a_i\in \N$, and, if so, what is the minimum value for $M_n$.  Since 1909, when Hilbert proved that such a bound exists, an enormous literature has developed, largely devoted to determining $M_n$.  There is also a substantial literature devoted to variants of Waring's problem.  Kamke proved \cite{Ka} a generalization of the theorem in which $n$th powers are replaced by general polynomials.  
In a series of papers \cite{W1,W2,W3}, Wooley solved Waring's problem for vector-valued polynomials.
Siegel  \cite{Siegel,Siegel1} treated the case of rings of integers in number fields, and since then, many papers have analyzed Waring's problem for a wide variety of rings, for instance \cite{Bi, Ca, Vo, GV, LW, Ch, Ellison}.   
Also, there has been a flurry of recent activity on ``Waring's problem for groups''; the typical problem here is to prove that every element in $G$  is a product of a small number of $n$th powers of elements of $G$ (see, for instance, \cite{Sh,LST,AGKS,GT} and the references therein.)

This paper explores the view that algebraic groups are the natural setting for Waring's problem.  To this extent, it resembles the work on Waring's problem for groups of Lie type.  The work on the polynomial-valued and vector-valued variants of Waring's problem also fit naturally in this framework.  We will consider morphisms of varieties (resp. schemes) $f\colon \A^1\to G$ defined over  a field (resp. a number ring) and look at bounded generation of the groups generated by the images.  

The strategy is developed in \S2 for unipotent algebraic groups over fields of characteristic $0$ which are not formally real.  (Some justification for concentrating on the unipotent case is given in Lemma~\ref{non-generating} below and the following remarks.)
In \S3, we solve the unipotent version of Waring's problem for totally imaginary number rings.  
In \S4, we work over general characteristic $0$ fields and general number rings, but consider only the ``easier Waring's problem'', in which one is allowed to use inverses.  Our methods throughout are elementary.  The only input from analytic number theory is Siegel's solution of Waring's problem over number rings.

Unfortunately, in the original situation of Waring's problem, namely the ring $\Z$, the additive group $\bG_a$, and the morphism $f\colon \A^1\to \bG_a$ given by $f(x) = x^n$,
our results fall short of Hilbert's theorem; we can prove only the easier Waring's problem in this case, rather than the statement that every positive integer can be represented as a bounded sum of non-negative
$n$th powers.
The difficulty, of course, is the ordering on $\Z$.  It seems natural to ask whether, for unipotent groups over general number rings, one can characterize the  set which ought to be expressible as a bounded product of images.  In proving the easier Waring's problem, we simply avoid this issue.

\section{Generating subvarieties}

Throughout this paper, $K$ will always be a field of characteristic $0$, and $G$ will be an
algebraic group over $K$.  A \emph{variety} over $K$ will be a reduced separated scheme of finite type and, in particular, need not be connected.
A subvariety, closed subgroup, etc., will always be understood to be defined over $K$.

\begin{defn}
Let $G$ be an algebraic group over a field $K$.
A subvariety $X$ of $G$ is \emph{generating} if there exists $n$ such that every generic point of $G$ lies in the image of the  product map from $X^{\times n} := X\times\cdots\times X$ to $G$.  A finite collection $f_i\colon X_i\to G$ of morphisms is generating if the union of Zariski closures $\bigcup_i \overline{f(X_i)}$ is generating.
\end{defn}

We have the following necessary and sufficient condition for a subvariety to be generating.

\begin{prop}
\label{criteria}
Let $G$ be an algebraic group over $K$ and $Z\subseteq G$ a closed subvariety.
Then $Z$ is generating if and only if it satisfies the following two properties:
\begin{itemize}
\item[I.] $Z$ is not contained in any proper closed subgroup of $G$.
\item[II.] For every proper closed normal subgroup $H$ of $G$, the image of $Z\to G/H$ has positive dimension.
\end{itemize}

\end{prop}

We first prove the following technical lemma:

\begin{lem}
\label{XYX}
Let $K$ be algebraically closed.
Let $X, Y$ be irreducible closed subvarieties of $G$.  Assume:
\begin{itemize}

\item [(1)] $\dim (\overline{XX}) = \dim X$;

\item [(2)] $\dim(\overline{XYX}) = \dim X$.

\end{itemize}
Then there exists a closed subgroup $H$ of $G$ such that the following statements are true:
\begin{itemize}

\item [(i)] $X = xH = H x$ for all $x \in X(K)$; 

\item [(ii)] $Y \subset yH$ for some $y \in N_{G}(H)(K)$.

\end{itemize}

\end{lem}

\begin{proof}
As $X$ is irreducible, $X^{\times 2}$ is irreducible, with generic point $\eta$.  The closure $\overline{X^2}=\overline{XX}$ of its image in $G$ is therefore the closure of the image of 
$\eta$ in $G$ and thus
irreducible.  If $x\in X(K)$, then $x X$ and $X x$ are closed subvarieties of $X^2$ of dimension $\dim X=\dim X^2$, so $xX = X^2 = X x$.
Thus, $X x^{-1} = x^{-1} X$.

It follows that for $x_1,x_2\in X(K)$,
$$x_1^{-1}X = x_1^{-1}(x_2^{-1}X^2) = (x_1^{-1}X)(x_2^{-1}X) = (x_1^{-1} X^2)x_2^{-1} = X x_2^{-1} = x_2^{-1} X.$$
Defining $H := x^{-1} X$ for $x\in X(K)$, we see that $H$ does not depend on the choice of $x$
and, moreover, that $H^2 = H$.  As every $h\in H(K)$ can be written $x_1^{-1} x_2$, $x_1,x_2\in X(K)$, it follows that $h^{-1} = x_2^{-1} x_1 \in H(K)$.
Thus $H(K)$ is a subgroup, which since $K$ is algebraically closed implies that $H$ is a closed subgroup of $G$,
which implies (i).

For $y\in Y(K)$,  $\overline{XYX}$ is connected and contains $\overline{X y X}$, which has  dimension $\ge \dim X = \dim \overline{XYX}$.  Thus,
$$\overline{X Y X} = \overline{X y X} = \overline{H xyx H}$$
is connected and has dimension $\dim H$.  It follows that the double coset $H xyx H$ consists of a single left coset, so $xyx \in N_{G}(H)(K)$.
By (i), $x$ also normalizes $H$, and it follows that $y$ normalizes $H$.  Finally,
$$Y \subseteq \overline{x^{-1} X Y X x^{-1}} = \overline{(x^{-1} H x) y  (x^{-1} H x)} = \overline{H y H} = yH.$$

\end{proof}

Using this, we can prove Proposition~\ref{criteria}.

\begin{proof}

Clearly, if $Z\subseteq H\subsetneq G$, then the same is true for $\overline{Z^n}$, and if the image of $Z$ in $G/H$ is finite, the same is true for $\overline{Z^n}$.  This proves necessity of conditions I and II.

For the sufficiency, we may assume without loss of generality that $K$ is algebraically closed.
For all $z\in Z(K)$, $Z \overline{Z^n}\subseteq \overline{Z^{n+1}}$, so $\dim \overline{Z^n}$ is a bounded non-decreasing sequence of integers.  It therefore stabilizes for some $n$.  Let $X$ denote an irreducible component of $\overline{Z^n}$ of dimension
$\dim \overline{Z^n}$ and $Y$ any irreducible component of $Z$.  Then $\dim X \le \dim \overline{X^2}\le \dim \overline{Z^n}$
and $\dim X\le \dim \overline{XYX} \le \dim \overline{Z^n}$, so conditions (1) and (2) of Lemma~\ref{XYX} are satisfied.
Let $H$ be the closed subgroup of $G$ satisfying (i) and (ii).   As $H$ is a translate of $X$, it is irreducible.

If $H = G^\circ$, then $X$ is a connected component of $G$, which means $\overline{Z^m}$ is a union of components of $G$ whenever $m\ge n$.  Applying condition $I$ to subgroups of $G$ containing $G^\circ$, it follows that every generic point of $G$ lies
in $Z^m$ for some $m\ge n$, so $Z$ is generated, as claimed.
If $H\neq G^\circ$< then $\dim H < \dim G$.  If $H$ is normal in $G$, 
then the image of $Z$ in $G/H$ is finite, contrary to condition II.  If $H$ has normalizer $N\subsetneq G$, then $Y$ is contained in $N$.  Since $N$ does not depend on the choice of component $Y$, $Z\subseteq N$, contrary to condition I.

\end{proof}

Henceforth, we assume $G$ is connected.  We are interested in generating collections of morphisms $\A^1\to G$.
By a theorem of Chevalley, Barsotti, and Rosenlicht \cite{Rosenlicht}, every connected algebraic group $G$ has a closed normal subgroup $H$ which is a linear algebraic group and
such that $G/H$ is an abelian variety.  Every map from a rational curve to an abelian variety is trivial.  Thus, unless $G$ is a linear algebraic group, it is impossible for any collection of 
morphisms $\A^1\to G$ to be generating.

Let $R$ and $U$ denote respectively the radical and the unipotent radical of $G$.  

\begin{lem}
\label{non-generating}
If $U\subsetneq R$, then there does not exist a generating set of morphisms $\A^1\to G$.
\end{lem}

\begin{proof}
It suffices to prove that there is no generating set of morphisms from $\A^1$ to the connected reductive group $G/U$.
Thus, we may assume without loss of generality that $G$ is connected reductive.  If the radical $R$ is non-trivial, then the inclusion map $R\to G$ induces an isogeny of tori
$R\to G/[G,G]$, so it suffices to prove that there no generating set of morphisms from $\A^1$ to a non-trivial torus $T$.  Without loss of generality, we may assume that $K=\bar K$.
Thus we may replace $T$ by a quotient isomorphic to the multiplicative group, and it suffices to prove there is no non-constant morphism of curves from $\A^1$ to $\A^1\setminus \{0\}$.
At the level of coordinate rings, this is the obvious statement that every $K$-homomorphism from $K[t,t^{-1}]$ to $K[x]$ maps $t$ to an element of $K^*$, or, equivalently, the fact that $K[x]^* = K^*$.

\end{proof}

We need only consider, then, the case that $G$ is the extension of a semisimple group by a unipotent group, both connected.
The semisimple case is perhaps even more interesting, but we know that, at least for $\SL_2$, we cannot always expect bounded generation, since, for example, $\SL_2(\Z)$ and $\SL_2(\Z[i])$ do not have bounded generation by elementary matrices \cite{GS, Tavgen}.

Since the characteristic of $K$ is $0$, if $G$ is unipotent, it is 
necessarily connected \cite[IV, \S2, Prop.~4.1]{DG}.  The derived group $G'$ is then likewise unipotent
\cite[IV, \S2, Prop.~2.3]{DG}, and therefore connected.  The quotient $G/G'$ is unipotent \cite[IV, \S2, Prop.~2.3]{DG},
and commutative
and is therefore a vector group \cite[IV, \S2, Prop.~4.2]{DG}.  The non-abelian Galois cohomology group $H^1(K,G')$ vanishes \cite[III~Prop.~6]{Serre}, so the cohomology sequence for the short
exact sequence $1\to G'\to G\to G/G'\to 1$ \cite[I~Prop.~38]{Serre} implies $(G/G')(K) = G(K)/G'(K)$.  We identify these groups.  We do not distinguish between
closed (vector) subgroups of $G/G'$ at the level of algebraic groups over $K$ and the corresponding $K$-subspaces of the vector space $(G/G')(K)$.  If $V$ is a subspace of $G/G'$, we denote by $VG'$ the inverse image of $V$ in $G$, regarded as an algebraic group.

\begin{lem}
Let $G$ be a connected unipotent algebraic group, and let $H$ be a proper closed subgroup of $G$.
Then the normalizer of $H$ in $G$ is strictly larger than $H$.
\end{lem}

\begin{proof}
We use induction on $\dim G$.  The case $\dim G=1$ is trivial since this implies $G$ is commutative,
so the normalizer of every subgroup is all of $G$.  For general unipotent $G$, the fact that the lower central series goes
to $\{1\}$ implies that the center $Z$ of $G$ is of positive dimension.  If $Z$ is not contained in $H$, then $ZH \subseteq N_G(H)$
is strictly larger than $G$.  Otherwise, replacing $G$ and $H$ by $G/Z$ and $H/Z$ respectively, we see that $H/Z$ is normal in
$N/Z$ for some $N\subseteq G$ strictly larger than $H$, so $H$ is normal in $N$.
\end{proof}

\begin{prop}
If $G$ is a unipotent group over $K$, then every proper closed subgroup $H$ of $G$
is contained in a normal subgroup $N$ of codimension $1$ in $G$ which contains the derived group $G' $ of $G$.
\end{prop}

\begin{proof}
As $K$ is of characteristic zero, $G$ is connected.  If $H=\{e\}$, the proposition asserts that $G$ contains a codimension $1$ normal subgroup containing $G'$.  As $H$ is a proper subgroup, $G$ is non-trivial, so $G/G'$ is non-trivial, and the proposition amounts to the obvious statement that every non-trivial vector group contains a normal subgroup of codimension $1$.

For the general case,
applying the previous lemma, we can replace $H$ by a strictly larger
group $N_G(H)$ unless $H$ is normal in $G$.  This operation can be repeated only finitely many times, since $N_G(H)$, being strictly larger than $H$,
must be of strictly higher dimension (since every closed subgroup of a unipotent group is unipotent and therefore connected).
Thus, we may assume $H$ is normal in $G$.  Then $G/H$ is unipotent.   Replacing $G$ and $H$ by $G/H$ and $\{e\}$ respectively, we are done.
\end{proof}

From Proposition~\ref{criteria},
we deduce that for unipotent groups, we have the following simple criterion.

\begin{lem}
Let $G$ be a unipotent group over  $K$.  A subvariety $X$ of $G$ (resp. a set  $\{f_1,\ldots,f_n\}$ of morphisms  $\A^1\to G$) is generating if and only if for each proper subspace $V\subsetneq G/G'$, such that the projection of $X$ to $G/VG'$ is of positive dimension (resp. the composition of some $f_i$ with the projection $G\to G/VG'$ is non-constant.)
\end{lem}

Note that the question of whether a set of morphisms $f_i$ is generating depends only on the set of compositions $\bar f_i$ of $f_i$ with the quotient map $G\to G/G'$.  It is also invariant under left or right translation of the $f_i$ by any element of $G(K)$.

\begin{lem}
If $\{f_1,\ldots,f_n\}$ is not generating, then for all positive integers $N$,
$$(f_1(K)\cup\cdots\cup f_n(K))^N \subsetneq G(K).$$
\end{lem}

\begin{proof}
The image of $(f_1(K)\cup\cdots\cup f_n(K))^N$ in $(G/V G')(K)$ is the same as the image of $\{f_1(0),\ldots,f_n(0)\}^N$ and is therefore a finite subgroup of an infinite group.
\end{proof}

We record the following lemma, which will be needed later.
\begin{lem}
\label{bracket}
Let $G$ be a  unipotent group over $K$, $G'$ its derived group and $G''$ the derived group of $G'$.  If $V$ is a proper subspace of $G'/G''$,
then there exists a dense open subvariety $U_1\subset G$ and for all $\gamma_1\in U_1(K)$, a dense open subvariety $U_2$ of the form $G\setminus WG'$,  such that for all $\gamma_2\in U_2(K)$,
$[\gamma_1,\gamma_2]$ does not lie in $V G''(K)$.
\end{lem}

\begin{proof}
Without loss of generality, we assume $K$ is algebraically closed.
As the characteristic of $K$ is $0$, $G$ and $G'$ are connected, so $G'/G''$ is connected.  The composition $G\times G\to G'/G''$ of the commutator map and the quotient map has the property that its image generates $G'/G''$ and is therefore not contained in $V$.  It follows that the inverse image $U$ of the complement of $V$ is dense and open in $G\times G$.  By Chevalley's theorem, the projection of $U\subseteq G\times G$ onto the first factor $G$ is a constructible set containing the generic point; it therefore contains an open dense $U_1$.  The fiber over any point $\gamma_1\in U_1(K)$ is non-empty.
The condition on $\gamma_2$ that $[\gamma_1,\gamma_2]\not\in VG'$ is linear on the image of $\gamma_2$ in $G/G'$ and is satisfied for at least one $\gamma_2$,
so $U_2$, defined by the condition $[\gamma_1,\gamma_2]\not\in VG'$ satisfies the properties claimed.

\end{proof}

\section{The unipotent Waring Problem over nonreal fields}

\begin{defn}
We say a field $K$ is \emph{nonreal} if it is of characteristic zero but not formally real (i.e., $-1$ is a sum of squares in $K$).
\end{defn}

The main theorem of the section is the following:

\begin{thm}
\label{ti-field}
If $G$ is a unipotent algebraic group over a nonreal field $K$ and $\{f_1,\ldots,f_n\}$ is a generating set of $K$-morphisms $\A^1\to G$, then for some positive integer $M$,
$$(f_1(K)\cup\cdots\cup f_n(K))^M = G(K).$$

\end{thm}

The proof occupies the rest of this section.  It depends on the following two propositions:

\begin{prop}
\label{ti-field-vg}
Theorem \ref{ti-field} holds when $G$ is a vector group.
\end{prop}

\begin{prop}
\label{ti-field-ind}
Under the hypotheses of Theorem~\ref{ti-field}, there exists an integer $m$, a sequence of elements $g_1,\ldots,g_m\in G(K)$, a sequence of positive integers $k_1,\ldots,k_m$,
for each $i\in \{1,\ldots,m\}$, a sequence of integers $\ell_{i,1},\ldots,\ell_{i,k_i}\in [1,n]$ and of $a_{i,j},b_{i,j}\in K$, such that for each $i\in \{1,\ldots,m\}$, the $K$-morphisms $h_1,\ldots,h_m\colon \A^1\to G$ defined by
$$h_i(x) := g_if_{\ell_{i,1}}(a_{i,1}x+b_{i,1})\cdots f_{\ell_{i,k_i}}(a_{i,k_i}x+b_{i,k_i})$$
map $\A^1\to G'$ and as morphisms to $G'$ are generating.
\end{prop}

Assuming both propositions hold, we can prove Theorem~\ref{ti-field} by induction on dimension.  If $G$ is commutative, then Proposition~\ref{ti-field-vg} applies.
Otherwise, we apply Proposition~\ref{ti-field-ind} to construct $h_1,\ldots,h_m$.  
Letting $\bar f_i$ denote the composition of $f_i$ with $G\to G/G'$,
Proposition~\ref{ti-field-vg} asserts that
every element of $G(K)/G'(K) = (G/G')(K)$ is represented by a bounded product of elements
of $\bar f_1(K)\cup\cdots\cup \bar f_n(K)$.  For each $g_i$, there exists $g'_i$, which is a bounded product of
elements of $f_1(K)\cup\cdots\cup f_n(K)$ and lies in the same $G'(K)$-coset of $G(K)$.  Defining
$$h'_i(x) := g'_if_{\ell_{i,1}}(a_{i,1}x+b_{i,1})\cdots f_{\ell_{i,k_i}}(a_{i,k_i}x+b_{i,k_i}),$$
it suffices to prove that every element of $G'(K)$ is a bounded product of elements of $h'_1(K)\cup\cdots\cup h'_m(K)$.
As the $h_i$ are generating for $G'$, the same is true for $h'_i$, and the theorem follows by induction.
Thus, we need only prove Propositions \ref{ti-field-vg} and \ref{ti-field-ind}.  

To prove Proposition~\ref{ti-field-vg}, we begin with a special case.
\begin{prop}
\label{twisted}
If $K$ is a characteristic zero field which is not formally real and $d$ is a positive integer, there exists an integer $N > 0$ such that
every vector in $K^d$ is a sum of elements of 
$$X^d_1 := \{(x,x^2,\ldots,x^d)\mid x\in K\}.$$
\end{prop}

\begin{proof}
For each integer $k > 0$, let
$$X^d_k := \underbrace{X^d_1+\cdots+X^d_1}_k.$$
Thus $X^d_1\subseteq X^d_2\subseteq\cdots$, and we denote by $X^d$ the limit $\bigcup_i X^d_i$.  Clearly $X^d_i + X^d_j = X^d_{i+j}$ and $X^d_i X^d_j \subseteq X^d_{ij}$.  Taking unions, $X^d$ is a semiring.  Let
$p^{d+1}\colon X^{d+1}\to X^d$ denote the projection map onto the first $d$ coordinates and $\pi^{d+1}\colon X^{d+1}\to K$ the projection map onto the last coordinate.
In particular, $p^{d+1}(X^{d+1}_k) = X^d_k$ for all positive integers $k$.

It is a theorem \cite[Theorem 2]{Ellison} that for each positive integer $d$ there exists $M$ such that
every element in $K$ is the sum of $M$ $d$th powers of elements of $K$.  

We proceed by induction on $d$.  The theorem is trivial for $d=1$.  Assume it holds for $d$ and choose $M$ large enough that
$X^d_M = X^d = K^d$ and every element of $K$ is the sum of $M$ $(d+1)$st powers.   In particular, 
\begin{equation}
\label{onto}
\pi^{d+1}(X^{d+1}_M) = K.
\end{equation}
If $X^{d+1}_M\subsetneq X^{d+1}_{M+1}$, then choosing $w\in X^{d+1}_{M+1}\setminus X^{d+1}_M$, there exists an element $v\in X^{d+1}_M$ such that $p^{d+1}(v) = p^{d+1}(w)$.  If $u\in X^{d+1}_M$ is chosen with $p^{d+1}(u) = -p^{d+1}(v)$, then either $u+v$ or $u+w$ is non-zero, and either
way there exists an element $t\in X^{d+1}_{2M+1}\setminus \{0\}$ with $p^{d+1}(t)=0$.  By (\ref{onto}), 
$$\{(0,\ldots,0,x)\mid x\in K\}\subset X^{d+1}_{M(2M+1)},$$
so by the induction hypothesis, $X^{d+1}_{M+M(2M+1)} = K^{d+1}$, and we are done.  We may therefore assume
$X^{d+1}_M = X^{d+1}_{M+1}$, which implies $X^{d+1}_M = X^{d+1}$.  Moreover, if $p^{d+1}$ fails to be injective, the same argument applies,
so we may assume that $p^{d+1}$ is an isomorphism of semirings, and therefore an isomorphism of rings (since the target $K^d$ is a ring).

Thus, we can regard $\pi^{d+1} \circ (p^{d+1})^{-1}$ as a ring homomorphism $\phi\colon K^d\to K$.  If $e_i = (0,\ldots,0,1,0,\ldots,0)$, then
$\phi(e_i)$ maps to an idempotent of $K$, which can only be $0$ or $1$.  Since $\phi(e_1+\cdots+e_d) = 1$, there exists  $i$ such that
$\phi(e_i) = 1$, and it follows that $\phi$ factors through projection onto the $i$th coordinate.  Thus, there exists a ring endomorphism $\psi \colon K\to K$
such that for all $(x_1,\ldots,x_{d+1})\in X^{d+1}$, $\psi(x_i) = x_{d+1}$.  As $(2,4,8,\ldots,2^{d+1})\in X^{d+1}$, $\psi(2^i) = 2^{d+1}$, which is absurd.
\end{proof}

We now prove Proposition~\ref{ti-field-vg}.

\begin{proof}
Let 
\begin{equation}
\label{P-vector}
f_j(x) = (P_{1j}(x),\ldots,P_{mj}(x)),
\end{equation}
where $d$ is the maximum of the degrees of the $P_{ij}$ for $1\le i\le m$, $1\le j \le n$.  Let $N$ be chosen as in Proposition~\ref{twisted}.  We write 
\begin{equation}
\label{P-formula}
P_{ij}(x) = \sum_{k = 0}^d a_{ijk}x^k
\end{equation}
for $1\le i\le m$, $1\le j \le n$. Given $c_1,\ldots,c_n$,
our goal is to find $x_{j\ell}\in K$, $1\le j \le n$, $1\le \ell \le N$ that satisfy the system of equations
\begin{equation}
\label{system}
\sum_{j,k,\ell} a_{ijk} x_{j\ell}^k = c_i,\ i=1,2,\ldots,m.
\end{equation}
By Proposition~\ref{twisted}, by choosing $x_{j \ell}$ suitably,
we can choose  the values
$$y_{jk} = \sum_{\ell = 1}^N x_{j \ell}^k$$
independently for $1\le j \le n$ and $1\le k \le d$, 
while $y_{j0} = N$ by definition.  Thus, we can rewrite the system of equations (\ref{system}) as
\begin{equation}
\label{linear-system}
\sum_{k=1}^d \sum_{j = 1}^n a_{ijk} y_{jk} = c_i - N\sum_{j = 1}^n a_{ij0},\ i=1,2,\ldots,m.
\end{equation}
This is always solvable unless there is a non-trivial relation among the linear forms on the left hand side in this system, i.e., a non-zero sequence  $b_1,\ldots,b_m$
such that
$$\sum_{i = 1}^m b_i a_{ijk} = 0$$
for all $j$ and $k\ge 1$.  If this is true, then
$$\sum_{i = 1}^m b_i P_{ij} = \sum_i b_i a_{ij0},\ j=1,\ldots,n.$$
In other words, defining $\pi(t_1,\ldots,t_n) := b_1 t_1 + \cdots + b_m t_m$, $\pi\circ f_j$ is constant for all $j$, contrary to assumption.
\end{proof}

Finally, we prove Proposition~\ref{ti-field-ind}.

\begin{proof}
Suppose we have already constructed $h_1,\ldots,h_r$.  Let $\bar h_i$ denote the composition of $h_i$ with the projection $G'\to G'/G''$.  Let $W$ denote the vector space spanned by the set $\{\bar h_i(t) - \bar h_i(0)\mid 1\le i\le r,\;t\in K\}$.  Suppose $W = G'/G''$.  Then for all proper subspaces $W'\subsetneq W$ there exist $i$ and $t$ such that $\bar h_i(0)$ and $\bar h_i(t)$ represent different classes in $W/W'$.  It follows that the composition of $h_i$ with the projection $G'\to G'/W'G''$ is non-constant, and therefore, $\{h_1,\ldots,h_r\}$ is a generating set of morphisms to $G'$.

Thus, we may assume that $W$ is a proper subspace of $G'/G''$.  We apply Lemma~\ref{bracket} to deduce the existence
of $\gamma_1\in G(K)$ and a proper closed subspace $V$ of $G/G'$ such that for all $\gamma_2\in G(K)\setminus VG'(K)$, 
the commutator of $\gamma_1$ and $\gamma_2$ is not in
$WG''(K)$.  Let $\bar f_i(x)$ denote the composition of $f_i(x)$ with the quotient map $G\to G/G'$.
As the $f_i$ are generating, there exists $i$ such that for all but finitely many values $x\in K$, $\bar f_i(x)$ is not in $V$.  Without loss of generality, we assume $i=1$.

By Proposition~\ref{ti-field-vg}, there exists a bounded product 
$$g=f_{s_1}(b_1)\cdots f_{s_M}(b_M),\;s_i\in \{1,2,\ldots,n\},\;b_i\in K$$
such that
$$\bar f_{s_1}(b_1)+\cdots +\bar f_{s_M}(b_M)=\bar\gamma_1.$$

We write
$$\bar f_1(x) = v_0 + xv_1 + \cdots + x^d v_d,$$
with $v_i \in K^m$.  
By Proposition~\ref{twisted}, there exist $a_1,\ldots,a_N\in K$, not all zero, such that
\begin{equation}
\label{a-condition}
\sum_{i=1}^n a_i^k = 0,\ k=1,2,\ldots,d.
\end{equation}
Thus,
$$\bar f_1(a_1x) + \bar f_2(a_2x) + \bar f_3(a_3x) + \cdots + \bar f_N(a_N x) = Nv_0,$$
which means that $x\mapsto f_1(a_1x)f_2(a_2x)\cdots f_N(a_nx)$ goes to a constant $G'$-coset.
Without loss of generality, we may assume $a_1\neq 0$.

Let $g_{r+1}$ be an element of $G(K)$ which can be realized as a product of at most $N$  values of $f_i$ at elements of $K$ and such that
\begin{equation}
\label{try1}
g_{r+1}f_{s_1}(b_1)\cdots f_{b_M}(b_M) f_1(a_1x)f_2(a_2x)\cdots f_N(a_nx)
\end{equation}
belongs to $G'(K)$ for $x=0$.  By Proposition~\ref{ti-field-vg}, such a $g_{r+1}$ exists.
We choose $h_{r+1}(x)$ to be either (\ref{try1}) or
\begin{equation}
\label{try2}
g_{r+1}f_1(a_1x)f_{s_1}(b_1)\cdots f_{s_M}(b_M) f_2(a_2x)\cdots f_N(a_nx),
\end{equation}
Either way, $h_{r+1}(0)\in G'(K)$.  By (\ref{a-condition}), $h_{r+1}(x)\in G'(K)$ for all $x\in K$.
Because the commutator $[f_1(a_1x),f_{s_1}(b_1)\cdots f_{s_M}(b_M)]$   lies in $WG''$
for at most finitely many $x$-values, at least one of (\ref{try1}) and (\ref{try2}) is non-constant (mod $WG''$).
By induction on the codimension of $W$, the proposition follows.

\end{proof}

\section{The unipotent Waring Problem over totally imaginary number rings}

In this section $K$ denotes a totally imaginary number field, $\cO$ its ring of integers, and $\cG$ a closed $\cO$-subscheme of the group scheme
$\cU_k$ of unitriangular $k\times k$ matrices.  Thus the generic fiber $\cG_K$ will be a closed subgroup of $\cU_k$ over $K$ and
therefore unipotent.  Moreover, there is a filtration of $\cG(\cO)$ by normal subgroups such that the successive quotients are finitely generated free abelian groups.
In particular, it is (by definition) a finitely generated torsion-free nilpotent group, whose Hirsch number is the sum of the ranks of these successive quotients.

A set $\{f_1,\ldots,f_n\}$ of $\cO$-morphisms $\A^1\to \cG$ is said to be generating if it is so over $K$.
The main theorem in this section is the following integral version of Theorem~\ref{ti-field}:

\begin{thm}
\label{ti-ring}
If  $\{f_1,\ldots,f_n\}$ is a generating set of $\cO$-morphisms $\A^1\to \cG$, then for some positive integer $M$,
$$(f_1(\cO)\cup\cdots\cup f_n(\cO))^M$$
is a subgroup of finite index in $\cG(\cO)$.
\end{thm}

We begin by proving results that allow us to establish that some power of a subset of a group $\Gamma$ gives a finite index subgroup of $\Gamma$.

\begin{lem}
Let $\Gamma$ be a group, $\Delta$ a finite index subgroup of $\Gamma$, $\Xi$ a subset of $\Gamma$, and $M$ a positive integer.
If $\Delta\subseteq \Xi^M$, then there exists $N$ such that $\Xi^N$ is a finite index subgroup of $\Gamma$.
\end{lem}

\begin{proof}
Without loss of generality, we assume that $\Delta$ is normal in $\Gamma$.
Consider the (finite set) 
$$\{(\bar m,\bar \gamma)\in \Z/M\Z\times \Gamma/\Delta\mid m\in \N,\;\gamma\in \Xi^m\}.$$
For each element, choose a pair $(m,\gamma)$, $m\in \N, \gamma\in \Xi^m$ representing it.  Choose $A$ to be greater than all values $m$ appearing in such pairs.  Let 
$N$ be a multiple of $M$ which is greater than $A+M$.  For all positive integers $k$, $\Xi^{kN}$ is a union of cosets of $\Delta$ in $\Gamma$ and does not depend on $k$.
The image of $\Xi^N$ in $\Gamma/\Delta$ is therefore a subset of a finite group and closed under multiplication.  It is therefore a subgroup, and the lemma follows.
\end{proof}

\begin{lem}
\label{Smaller}
Let $\Gamma$ be a finitely generated nilpotent group and $\Delta$ a normal subgroup of $\Gamma$.  Then every finite index subgroup $\Delta_0$ of $\Delta$ contains 
a finite index subgroup $\Delta_0^\circ$ which is normal in $\Gamma$.
\end{lem}

\begin{proof}
We prove there exists a function $f\colon \N\to \N$ depending only on $\Gamma$ such that for any normal subgroup $\Delta$ of $\Gamma$, every subgroup $\Delta_0$ of $\Delta$ of
index $n$ contains a normal subgroup of $\Gamma$ of index $\le f(n)$ in $\Delta$.
Replacing $\Delta_0$ with the kernel of the left action of $\Delta$ on $\Delta/\Delta_0$, we may assume without loss of generality that $\Delta_0$ is normal in $\Delta$.
We prove the claim by induction on the total number of prime factors of $n$.  

If $n=p$ is prime, it suffices to prove that there is an upper bound, independent of $\Delta$, on the number of
normal subgroups of $\Delta$ of index $p$.  This is true because intersecting a fixed central series of $\Gamma = \Gamma_0\triangleright \Gamma_1\triangleright\cdots$
with $\Delta$ gives a central series of $\Delta$, and every index $p$ normal subgroup of $\Delta$ is the inverse image in $\Delta$ of an index $p$ subgroup of the finitely generated
abelian group $\Delta/\Delta\cap \Gamma_1\subseteq \Gamma/\Gamma_1$.

If $n$ has $\ge 2$ prime factors, then for some prime factor $p$ of $n$, $\Delta_0$ is a normal subgroup of index $n/p$ of a normal subgroup $\Delta_p$ of index $p$ in $\Delta$.
By the induction hypothesis, $\Delta_p$ contains a normal subgroup $\Delta_p^\circ$ of $\Gamma$, of index $\le f(p)$ in $\Delta$.  The index of $\Delta_0\cap \Delta_p^\circ$ in $\Delta_p^\circ$
divides $n/p$, and applying the induction hypothesis, we deduce that the existence of a normal subgroup $\Delta_0^\circ$ of $\Gamma$ of index $\le \max_{i\le n/p} f(i)$ in $\Delta_p^\circ$.
\end{proof}

\begin{prop}
\label{fi-ext}
Let $\Gamma$ be a finitely generated nilpotent group, $\Delta$ a normal subgroup of $\Gamma$, $\Xi$ a subset of $\Gamma$ and $M_1,M_2$ positive integers such that
$\Xi^{M_1}$ contains a finite index subgroup of $\Delta$ and the image of $\Xi^{M_2}$ in $\Gamma/\Delta$ contains a finite index subgroup of $\Gamma/\Delta$.
Then there exists $M_3$ such that $\Xi^{M_3}$ is a finite index subgroup of $\Gamma$.
\end{prop}

\begin{proof}
Let $\langle \Xi \rangle$ denote the subgroup of $\Gamma$ generated by $\Xi$.  The intersection $\langle \Xi \rangle\cap \Delta$ is of finite index in $\Delta$, and the image $\langle \Xi \rangle \Delta/\Delta$
is of finite index in $\Gamma/\Delta$, so $\langle \Xi \rangle$ is of finite index in $\Gamma$.  As a subgroup of a finitely generated nilpotent group, it is also finitely generated and nilpotent.  Replacing $\Gamma$ and $\Delta$ by $\langle \Xi \rangle$ and
$\langle \Xi \rangle\cap \Delta$ respectively, we  assume without loss of generality that $\Xi$ generates $\Gamma$.

Replacing $\Xi$ with $\Xi^{M_1}$, we may assume $\Xi$ contains a finite index subgroup
$\Delta_0$  of $\Delta$.  By Lemma~\ref{Smaller}, we may assume that $\Delta_0$ is a normal subgroup of $\Gamma$.
Let $\bar\Xi$ denote the image of $\Xi$ in $\Gamma/\Delta_0$.  If $\bar\Xi^{N}$ is a finite index subgroup of $\bar \Gamma := \Gamma/\Delta_0$, then $\Xi^{N+1}$ is the inverse image of this subgroup in
$\Gamma$, and the proposition holds.  Replacing $\Gamma$, $\Delta$, $\Xi$ by $\bar\Gamma$, $\bar\Delta := \Delta/\Delta_0$, and $\bar\Xi$ respectively, we reduce to the case that $\Delta$ is finite.
We need only show that if $\Xi^{M_2}\Delta$ contains a finite index subgroup $\Gamma_0$ of $\Gamma$, then $\Xi^N$ is a finite index subgroup of $\Gamma$ for some $N$.

Replacing $\Xi$ by $\Xi^{M_2}\cap \Gamma_0 \Delta$, we may assume that $\Xi\subseteq \Gamma_0\Delta$ and $\Xi$ meets every fiber of $\pi\colon \Gamma_0\Delta\to \Gamma_0\Delta/\Delta$.  In particular, $\Xi$ contains an element of $\Delta$, so replacing $\Xi$ with $\Xi^{|\Delta|}$, we may assume $\Xi$ contains the identity, so 
\begin{equation}
\label{chain}
\Xi\subseteq \Xi^2\subseteq\Xi^3\subseteq\cdots
\end{equation}
For $i$ a positive integer, let $m_i$ denote the maximum over all fibers of $\pi$ of the cardinality of the intersection of the fiber with $\Xi^i$.
Thus, the intersection of every fiber of $\pi$ with $\Xi^{i+1}$ is at least $m_i$.  Since fiber size is bounded above by $\Delta$, the sequence (\ref{chain}) must eventually stabilize.
Replacing $\Xi$ with a suitable power, we have $\Xi^2 = \Xi$.  Thus $\Xi$ is closed under multiplication.  As $\Xi$ meets every fiber of $\pi$ in the same number of points,
$\gamma\in \Xi$ implies $\gamma (\Xi\cap \pi^{-1}(\pi(\gamma)^{-1})) = \Xi\cap \Delta_0$, which implies $\gamma^{-1}\in \Xi$.  Thus, $\Xi$ is a subgroup of $\Gamma$ of bounded index.


\end{proof}

Next, we prove a criterion for a subgroup of $\cG(\cO)$ to be of finite index.

\begin{prop}
Let $\Gamma\subseteq \cG(\cO)\subset \cG(K)$.  Then the Hirsch number $h \Gamma$ of $\Gamma$ satisfies
\begin{equation}
\label{Hirsch}
h\Gamma\le [K:\Q]\dim \cG_K.
\end{equation}
If equality holds in (\ref{Hirsch}), then $\Gamma$ is of finite index in $\cG(\cO)$.

\end{prop}

\begin{proof}
Hirsch number is additive in short exact sequences.  Let $G_0 := \cG_K$, and let $G_0\supsetneq G_1 \supsetneq \cdots\supsetneq G_k = \{1\}$ be a central series.
Then we have a decreasing filtration of $\Gamma$ by $\Gamma_i := \Gamma\cap G_i(K)$, and each quotient $\Gamma_i/\Gamma_{i+1}$ is a free abelian subgroup
of $G_i(K)/G_{i+1}(K) \cong K^{\dim G_i/G_{i+1}}$.  Every free abelian subgroup of $K^r$ has rank $\le r[K:\Q]$ with equality if and only if it is commensurable with $\cO^r$.
This implies (\ref{Hirsch}).

Applying the same argument to $\cG(\cO)$, we get 
$$h\Gamma\le h\cG(\cO)\le [K:\Q]\dim \cG_K.$$
If equality holds in (\ref{Hirsch}), then $\cG(\cO)_i/\cG(\cO)_{i+1}$ and its subgroup
$\Gamma_i/\Gamma_{i+1}$ are commensurable, and this implies that $\Gamma$ is of index
$$\prod_{i=0}^{k-1} [\cG(\cO)_i/\cG(\cO)_{i+1}:\Gamma_i/\Gamma_{i+1}]<\infty$$
in $\Gamma$.
\end{proof}

We prove Theorem~\ref{ti-ring} by showing that $h(f_1(\cO)\cup\cdots\cup f_n(\cO))^M$ contains a subset which is a group of Hirsch number 
$[K:\Q]\dim \cG_K$.
We first treat the commutative case.

\begin{prop}
\label{ti-ring-comm}
Theorem~\ref{ti-ring} holds if $\cG$ is commutative.
\end{prop}

\begin{proof}
First we claim that for all $d>0$ there exist integers $L,M>0$ such that
$$L \cO^d \subseteq\{(x_1+\cdots+x_M,x_1^2+\cdots+x_M^2,\ldots,x_1^d+\cdots+x_M^d)\mid x_i\in \cO\}.$$
Since this is of finite index in $\cO^d$, replacing $M$ by a larger integer (also denoted $M$), we can guarantee that every element
in the group generated by $\{(x,x^2,\ldots,x^d)\mid x\in \cO\}$ can be written as a sum of $M$ elements.

To prove the claim, we use Proposition~\ref{twisted} to show that each basis vector $e_i$ is a sum of $M$ elements of 
$\{(x,x^2,\ldots,x^d)\mid x\in K\}$.  Replacing each $x$ in the representation of $e_i$ by $Dx$ for some sufficiently divisible positive integer $D$,
it follows that each $k_i e_i$ can be written as a sum of $M$ elements of $\{(x,x^2,\ldots,x^d)\mid x\in \cO\}$ for suitable positive integers $k_i$.

For each $\alpha \in \cO$, we see from the $(i - 1)$th difference of $\alpha^i$ (see \cite[Theorem 1, p.267]{Wright}) that
\begin{align}
\label{alternate}
i!\alpha = \sum_{m = 0}^{i - 1}(-1)^{m} \binom{i - 1}{m} \left(\alpha + m\right)^i - \dfrac{1}{2}(i - 1)i!.
\end{align}
Thus every element of $i!\cO$ is in the subring $\cO^{(i)}$ of $\cO$ generated by $i$th powers of elements of $\cO$. A theorem of Siegel (see \cite[Theorem VI]{Siegel1}) implies that there exist $N_1, N_2, \ldots$ such that every element of $\cO^{(i)}$ is a sum of $N_i$ $i$th powers of elements of $\cO$. Thus every element of $i!\cO$ is a sum of $N_i$ $i$th powers of elements of
$\cO$, and therefore every element of $i!k_i \cO e_i$ is a sum of $MN_i$ elements of $\{(x,x^2,\ldots,x^d)\mid x\in \cO\}$.

Letting $L$ denote a positive integer
divisible by $1!k_1, 2!k_2\ldots, d!k_d$, and replacing $M$ by $M(N_1+\cdots+N_d)$, we can write every element of $L \cO^d$ as a sum of $M$ elements of 
$\{(x,x^2,\ldots,x^d)\mid x\in \cO\}$.

Restricting $f_j$ to the generic fiber, we can write it as a vector of polynomials (\ref{P-vector}) with the $P_{ij}$ given by (\ref{P-formula}).
We can solve the system (\ref{system}) of equations in $\cO$ whenever we can solve (\ref{linear-system}) in $y_{jk}\in L\cO$.
This system is always solvable in $K$, so it is solvable in $L\cO$ whenever the  $c_i - N\sum_j a_{ij0}$ is sufficiently divisible.  Thus, there exists an integer $D$ such that if $N$ and the $c_i$ are divisible by $D$ and $N$ is sufficiently large, then $(c_1,\ldots,c_m)$ is a sum of $N$ terms each of which belongs to $(f_1(\cO) \cup \cdots \cup f_n(\cO))$.  Let $\Lambda := D\cO^m$.

Now, $\Lambda\subseteq \cG(\cO)\subset \cG(K)\cong K^m$.  As $\cG(\cO)$ has a finite filtration whose quotients are finitely generated free abelian groups, it must contain $\Lambda$ as a subgroup of finite index.  Defining
$$X_i := \underbrace{\bigl(f_1(\cO)\cup\cdots\cup f_n(\cO)\bigr)+\cdots+\bigl(f_1(\cO)\cup\cdots\cup f_n(\cO)\bigr)}_{iN},$$
we have $\Lambda\subseteq X_i\subseteq \cG(\cO)$ for all $i\ge 1$, and $X_{i+1} = X_1 + X_i$.
It follows that $X_{i+1}$ contains every $\Lambda$-coset in $\cG(\cO)$ represented by any element of $X_i$, and therefore the sequence $X_1,X_2,\ldots$ stabilizes to a subgroup of $\cG(\cO)$ of rank $m[K:\Q]$ and of finite index in $\cG(\cO)$.
\end{proof}

Now we prove Theorem~\ref{ti-ring}
\begin{proof}
We first observe that Proposition~\ref{ti-field-ind} remains true over $\cO$; more precisely,  assuming that the morphisms $f_i$ are defined over $\cO$,  the elements $g_i$ can be taken to be in $\cG(\cO)$ and $a_{i,j}, b_{i,j}\in \cO$, so the morphisms $h_i$ are defined over $\cO$.
Instead of using Proposition~\ref{ti-field-vg}, we use Proposition~\ref{ti-ring-comm}.
The image $\bar\gamma_1$ of the element $\gamma_1$ guaranteed by Lemma~\ref{bracket} may not lie in the lattice $\Lambda\subset G(K)/G'(K)\cong K^m$, but some positive integer multiple of $\bar\gamma_1$ will do so, and the property of $\gamma_1$ with respect to $V$ is unchanged when it is replaced by a non-trivial power.  The elements $a_i$ guaranteed by Proposition~\ref{twisted} may not lie in $\cO$, but again we can clear denominators by multiplying by a suitable positive integer.
The element $g_{r+1}$ will exist as long as $N v_0\in \Lambda$.  This can be guaranteed by replacing $N$ with a suitable positive integral multiple.

Now we proceed as in the proof of Theorem~\ref{ti-field}, using induction on $\dim G$.  By the induction hypothesis, there exists $N$ such that $(\bigcup_i h_i(\cO))^N$ contains a subgroup of
$G'(K)$ of Hirsch number $[K:\Q]\dim G'$.  On the other hand, by Proposition~\ref{ti-ring-comm}, there exists a bounded power of $(\bar{f}_1(\cO)\cup\cdots\cup \bar{f}_n(\cO))$ which contains  a subgroup of
$(G/G')(K)$ of Hirsch number $[K:\Q]\dim G/G'$. Here for each $1 \le i \le n$, $\bar{f}_i(x)$ denotes the composition of $f_i(x)$ with the quotient map $G\to G/G'$. The theorem follows from Proposition~\ref{fi-ext} and the additivity of Hirsch numbers.

\end{proof}

\section{The easier unipotent Waring problem}

We recall that the classical ``easier Waring problem'' \cite{Wright} is to prove that for every positive integer $n$ there exists $m$ such that every integer can be written in the form $\pm a_1^n \pm \cdots\pm a_m^n$, $a_i\in \Z$,
and to determine the minimum value of $m$ for each $n$.

In this section, we prove unipotent analogues of the easier  Waring problem for arbitrary fields of characteristic zero and rings of integers of arbitrary number fields:

\begin{thm}
\label{easier-Waring-for-fields}
If $G$ is a unipotent algebraic group over a field $K$ of characteristic zero and $\{f_1,\ldots,f_n\}$ is a  generating set of $K$-morphisms $\A^1\to G$, then for some positive integer $M$,
\begin{align*}
\left(\bigcup_{e_1, \ldots, e_n \in \{\pm 1\}} (f_1(K)^{e_1}\cup\cdots\cup f_n(K)^{e_m})\right)^M = G(K).
\end{align*}

\end{thm}

\begin{thm}
\label{easier-Waring-ring}

Let $K$ be a number field, $\cO$ its ring of integers, and $\cG$ a closed $\cO$-subscheme of the group scheme
$\cU_k$ of unitriangular $k\times k$ matrices.
If  $\{f_1,\ldots,f_n\}$ is a generating set of $\cO$-morphisms $\A^1\to \cG$, then for some positive integer $M$,
$$\left(\bigcup_{e_1, \ldots, e_n \in \{\pm 1\}}f_1(\cO)^{e_1}\cup\cdots\cup f_n(\cO)^{e_n}\right)^M$$
is a subgroup of bounded index in $\cG(\cO)$.

\end{thm}

The proof of Theorem \ref{easier-Waring-for-fields} depends on variants of Propositions ~\ref{ti-field-vg} and ~\ref{ti-field-ind}. 

\begin{prop}
\label{diagonal-equations-in-easier-W-S3}

If $K$ is a field of characteristic zero and $d$ is a positive integer, there exists an integer $N > 0$ such that $K^d$ can be represented as
\begin{align*}
K^d = \underbrace{(X_1^d + \cdots + X_1^d)}_N - \underbrace{(X_1^d + \cdots + X_1^d)}_N,
\end{align*}
where
\begin{align*}
X_1^d := \{ (x, x^2, \ldots, x^d)\;  | \; x \in K \}.
\end{align*}

\end{prop}

\begin{proof}

This is \cite[Theorem~3.2]{Im}.
\end{proof}

\begin{prop}
\label{easier-W-for-fields-vg}

Theorem \ref{easier-Waring-for-fields} holds when $G$ is a vector group.

\end{prop}

\begin{proof}

Let 
\begin{equation}
\label{f-j-S3}
f_j(x) = (P_{1j}(x),\ldots,P_{mj}(x)),
\end{equation}
where $d$ is the maximum of the degrees of the $P_{ij}$ for $1\le i\le m$ and $1\le j \le n$.  Let $N$ be chosen as in Proposition \ref{diagonal-equations-in-easier-W-S3}.  Writing 
\begin{equation}
\label{P-formula-S3}
P_{ij}(x) = \sum_{k = 0}^d a_{ijk}x^k
\end{equation}
for $1\le i\le m$ and $1\le j \le n$. Given $(c_1, \ldots, c_m)$, our goal is to find suitable $\epsilon_{j \ell} \in \{\pm 1\}$ and $x_{j\ell} \in K$ such that
\begin{align*}
(c_1, \ldots, c_m) = \sum_{\ell = 1}^{2N} \sum_{j = 1}^n \epsilon_{j\ell} f_j(x_{j\ell}).
\end{align*}
In light of Proposition \ref{diagonal-equations-in-easier-W-S3}, for each $1 \le j \le n$, one can let $\epsilon_{j \ell} = 1$ if $1 \le \ell \le N$, and let $\epsilon_{j \ell} = - 1$ if $N + 1 \le \ell \le 2N$. Thus the above system is equivalent to the system of equations
\begin{align}
\label{system-easier-W-for-fields-S3}
c_i = \sum_{j = 1}^n \sum_{k = 0}^d a_{ijk}\left(\sum_{\ell = 1}^Nx_{j\ell}^k - \sum_{\ell = N + 1}^{2N}x_{j\ell}^k\right), \; i = 1, \ldots, m.
\end{align}

By Proposition ~\ref{diagonal-equations-in-easier-W-S3}, by choosing $x_{j\ell} \in K$ suitably, we can choose the values
\begin{align*}
y_{jk} = \sum_{\ell = 1}^Nx_{j\ell}^k - \sum_{\ell = N + 1}^{2N}x_{j\ell}^k
\end{align*}
independently for $1 \le j \le n$ and $1 \le k \le d$, while $y_{j0} = 0$ by definition. Thus we can rewrite the system of equations (\ref{system-easier-W-for-fields-S3}) as
\begin{align}
\label{linear-system-easier-W-for-fields-S3}
\sum_{j = 1}^n \sum_{k = 1}^d a_{ijk}y_{jk} = c_i, \; \; i = 1, \ldots, m.
\end{align}

Arguing as in the proof of Proposition ~\ref{ti-field-vg}, we see that the above system of equations is always solvable unless $f_j$ is constant modulo some proper subspace $V$ of $\A^m$ for all $1 \le j \le n$, i.e., each
$\pi \circ f_j$ is constant,
where $\pi : \A^m \to \A^m/V$ is the canonical projection. This is impossible since the set of morphisms $\{f_1, \ldots, f_n\}$ is  generating.

\end{proof}

\begin{prop}
\label{induction-step-for-easier-W-for-fields-S3}

Under the hypotheses of Theorem \ref{easier-Waring-for-fields}, there exists an integer $m$, a sequence of elements $g_1,\ldots,g_m\in G(K)$, a sequence of positive integers $k_1,\ldots,k_m$,
for each $i\in \{1,\ldots,m\}$, a sequence of integers $\ell_{i,1},\ldots,\ell_{i,k_i}\in [1,n]$, a sequence of integers $e_{i, 1}, \ldots, e_{i, k_i} \in \{\pm 1\}$, and sequences of elements $a_{i,1},b_{i,1},\ldots,a_{i,k_i},b_{i,k_i}\in K$, such that for each $i\in \{1,\ldots,m\}$, the $K$-morphisms $h_1,\ldots,h_m\colon \A^1\to G$ defined by
$$h_i(x) := g_if_{\ell_{i,1}}(a_{i,1}x+b_{i,1})^{e_{i, 1}}\cdots f_{\ell_{i,k_i}}(a_{i,k_i}x+b_{i,k_i})^{e_{i, k_i}}$$
map $\A^1\to G'$ and as morphisms to $G'$ are  generating.

\end{prop}

\begin{proof}

Using Proposition \ref{easier-W-for-fields-vg}, and the same arguments as in Proposition ~\ref{ti-field-ind}, Proposition \ref{induction-step-for-easier-W-for-fields-S3} follows immediately.

\end{proof}

\begin{proof}[Proof of Theorem ~\ref{easier-Waring-for-fields}]

The proof of Theorem ~\ref{easier-Waring-for-fields} is the same as that of Theorem ~\ref{ti-field}. Using Propositions \ref{easier-W-for-fields-vg} and \ref{induction-step-for-easier-W-for-fields-S3}, we proceed as in the proof of Theorem ~\ref{ti-field}, using induction on $\dim(G)$.  Theorem ~\ref{easier-Waring-for-fields} follows immediately.

\end{proof}

Next, we prove an integral variant of Proposition \ref{diagonal-equations-in-easier-W-S3}, in greater generality than we need for Theorem~\ref{easier-Waring-ring}:

\begin{prop}
\label{diag-eqn-for-rings-in-easier-W-S3}

Let $\cO$ be any integral domain whose quotient field $K$ is of characteristic zero. For all positive integers $d$, there exist $\lambda\in \cO\setminus \{0\}$ and $N\in \Z_{>0}$ such that
\begin{align*}
\lambda\cO^d \subseteq \underbrace{(Y_1^d + \cdots + Y_1^d)}_N -  \underbrace{(Y_1^d + \cdots + Y_1^d)}_N,
\end{align*}
where
\begin{align*}
Y_1^d = \{(x, x^2, \ldots, x^d) \; | \; x \in \cO\}.
\end{align*}

\end{prop}

\begin{proof}

For each integer $k > 0$, set
\begin{align*}
Y_{k, k}^d = \underbrace{(Y_1^d + \cdots + Y_1^d)}_k -  \underbrace{(Y_1^d + \cdots + Y_1^d)}_k.
\end{align*}

Choose $N > 0$ as in Proposition \ref{diagonal-equations-in-easier-W-S3}. For each $1 \le m \le d$, the basis vector $e_m$ can be written in the form
\begin{align*}
e_m = \sum_{j = 1}^N(x_j, x_j^2, \ldots, x_j^d) - \sum_{j = 1}^N(y_j, y_j^2, \ldots, y_j^d)
\end{align*}
for some $x_j, y_j \in K$. Replacing each $x_j, y_j$ in the above representation by $\delta x_j, \delta y_j$ for some non-zero $\delta\in \cO$, it follows that there exists a non-zero $\kappa_m\in \cO$ such that
\begin{align}
\label{eqn1-diag-eqn-for-rings-in-easier-W-S3}
\kappa_me_m \in Y_{N, N}^d.
\end{align}

For each $\alpha \in \cO$, we apply (\ref{alternate}) to prove that
\begin{align*}
m!\kappa_m\cO e_m \subseteq Y_{2NN_m, 2NN_m}^d.
\end{align*}

Let $\lambda :=  d!\prod_{i=1}^d \kappa_i$.   Replacing $N$ by $2N(N1 + \cdots + N_d)$, we deduce that
\begin{align*}
\lambda\cO^d \subseteq Y_{N, N}^d.
\end{align*}

\end{proof}

The next result is a variant of Proposition \ref{ti-ring-comm}.

\begin{prop}
\label{easier-Waring-ring-comm}

Theorem~\ref{easier-Waring-ring} holds if $\cG$ is commutative.

\end{prop}

\begin{proof}

Let $\lambda, N$ be chosen as in Proposition \ref{diag-eqn-for-rings-in-easier-W-S3}. Restricting $f_j$ to the generic fiber, we can write it as a vector of polynomials \eqref{f-j-S3} with the $P_{ij}$ given by \eqref{P-formula-S3}. We can solve the system \eqref{system-easier-W-for-fields-S3}  of equations whenever we can solve the system \eqref{linear-system-easier-W-for-fields-S3} in $y_{jk} \in \lambda\cO$. This system is always solvable in $K$; so it is solvable in $\lambda\cO$ whenever the $c_i$ are sufficiently divisible. Thus there exists an integer $D$ such that if the $c_i$ are divisible by $D$ and $N$ is sufficiently large, then $(c_1, \ldots, c_m)$ is a sum of $N$ terms, each of which belongs to $\bigcup_{e_1, \ldots, e_n \in \{\pm 1\}} (e_1f_1(\cO) \cup \cdots \cup e_nf_n(\cO))$. Let  $\Lambda := D\cO^m$. 

Set
\begin{align*}
U = \bigcup_{e_1, \ldots, e_n \in \{\pm 1\}} (e_1f_1(\cO) \cup \cdots \cup e_nf_n(\cO)).
\end{align*}
For each $i \ge 1$, define
\begin{align*}
X_i = \underbrace{U + \cdots + U}_{\text{$iN$ copies of $U$}}.
\end{align*}

We have $\Lambda \subseteq X_i \subseteq \cG(\cO) \subset \cG(K) \cong K^m$ for all $i \ge 1$, and $X_{i + 1} = X_1 + X_i$. Using the same arguments as in the proof of Proposition \ref{ti-ring-comm}, $\Lambda$ is a subgroup of finite index in $\cG(\cO)$, and therefore the sequence $X_1, X_2, \ldots$ stabilizes to a subgroup of $\cG(\cO)$ of rank $m[K : \Q]$, and of finite index in $\cG(\cO)$.

\end{proof}

We now prove Theorem \ref{easier-Waring-ring}.

\begin{proof}[Proof of Theorem \ref{easier-Waring-ring}]

We first observe that Proposition \ref{induction-step-for-easier-W-for-fields-S3} remains true over $\cO$; more precisely,  assuming that the morphisms $f_i$ are defined over $\cO$,  the elements $g_i$ can be taken to be in $\cG(\cO)$ and $a_{i,j}, b_{i,j}\in \cO$, so the morphisms $h_i$ are defined over $\cO$.

Now we proceed as in the proof of Theorem \ref{easier-Waring-for-fields}, using induction on $\dim G$.  By the induction hypothesis, there exists an integer $N > 0$ such that $(\bigcup_{e_1, \ldots, e_n \in \{\pm 1\}} (h_1(\cO)^{e_1} \cup \cdots \cup h_n(\cO)^{e_n}))^N$ is a subgroup of
$G'(\cO)$ of Hirsch number $[K:\Q]\dim G'$.  On the other hand, by Proposition~\ref{easier-Waring-ring-comm}, there exists a bounded power of $\bigcup_{e_1, \ldots, e_n \in \{\pm 1\}} (\bar f_1(\cO)^{e_1} \cup \cdots \cup \bar f_n(\cO)^{e_n})$ which is a subgroup of
$(G/G')(\cO)$ of Hirsch number $[K:\Q]\dim G/G'$. Here for each $1 \le i \le n$, $\bar f_i(x)$ denotes the composition of $f_i(x)$ with the quotient map $G\to G/G'$. The theorem follows by Proposition~\ref{fi-ext}  and the additivity of Hirsch numbers.

\end{proof}

\end{document}